\documentclass[10pt,leqno]{amsart}
\usepackage{amsfonts}
\usepackage[a4paper,inner=1in,outer=1in,vmargin=1in,marginparwidth=1in]{geometry}
\usepackage{amsmath,amsthm,amsfonts,amssymb}
\usepackage{bbm}
\usepackage{cite}
\usepackage{color}
\usepackage{enumerate}
\usepackage{physics}
\makeatletter
\@namedef{subjclassname@2020}{\textup{2020} Mathematics Subject Classification}
\makeatother
\newtheorem{theo}{Theorem}[section]
\newtheorem{defi}[theo]{Definition}
\newtheorem{lemm}[theo]{Lemma}
\newtheorem{remark}[theo]{Remark}

\newtheorem{prop}[theo]{Proposition}

\DeclareMathOperator{\Span}{span}
\numberwithin{equation}{section}
\allowdisplaybreaks

\begin{document}
	\title{ Integrable Modules of Map Full toroidal Lie algebras } 
	\author[P. Bisht]{Pradeep Bisht}
        \address{Pradeep Bisht:  Harish-Chandra Research Institute, A CI of Homi Bhabha National Institute, Chhatnag Road, Jhunsi, Prayagraj (Allahabad) 211 019 India}
	\email{pradeepbisht@hri.res.in, pradeepbishthri@gmail.com}
         \author[P. Batra]{Punita Batra*}   
         \address{Punita Batra: Harish-Chandra Research Institute, A CI of Homi Bhabha National Institute,  Chhatnag Road, Jhunsi, Prayagraj (Allahabad) 211 019 India }
         \email{batra@hri.res.in}
         \thanks{$^\star$  Corresponding author}


\date{}
	\begin{abstract}
		In this paper, we study the irreducible objects of the category $ \mathcal{C}_{fin} $ of integrable representations for Map full Toroidal Lie algebras with finite dimensional weight spaces. These representations turn out to be single point evaluation modules and hence are irreducible-integrable modules for the underlying full toroidal algebras.  
	\end{abstract} 
 \maketitle
	\section { Introduction}
 The Lie algebras and their representation theory is an interesting area of research, namely, due to their wide range of applicability in Mathematics and Physics. It is well known that the Virasoro Lie algebra acts on any (except when the level is negative dual coxeter number) highest weight module of affine Lie algebra through the use of Sugawara operators. Further, the construction of highest weight integrable modules of level 1 for affine Lie algebra is known from [15]. Thus the semi-direct product of affine Lie algebra and Virosoro algebra emerges to be an important Lie algebra in
 both Mathematics and Physics. Toroidal Lie algbera is a several variable analogue of the above important Lie algebra.   \par
Let $\mathfrak{\dot{g}}$ be a finite dimensional simple Lie algebra over the field of complex numbers $\mathbb{C}$, $ A = \mathbb{C}[t_{0}^{\pm{1}},...,t_{\nu}^{\pm{1}}]  $ $(\nu\geq 1)$ be the Laurent polynomial ring in $\nu+1$ commuting variables and $\mathcal{D}$ denotes the algebra of derivations of $ A $. Consider the universal central extension of the multiloop algebra $\mathcal{L}(\mathfrak{\dot g})=\mathfrak{\dot{g}}\otimes A$, denoted by $ \mathcal{\tilde L}(\mathfrak{\dot g})$ . On adding the $ \nu +1 $ degree derivations $d_{i}$'s to $ \mathcal{\tilde L}(\mathfrak{\dot g})$ we get the well known Toroidal Lie algebras, denoted by $  \mathcal{\hat L}(\mathfrak{\dot g})$. These algebras are multivariable analogue of the untwisted affine Lie algebras.
Note that for $\nu=0 $, these are just the untwisted affine Kac-Moody algebras. If we  add the algebra of derivations $ \mathcal{D}$ to $   \mathcal{\tilde L}(\mathfrak{\dot g}) $, we obtain the full Toroidal Lie algebra, denoted by $\tau$. Toroidal Lie algebra has been extensively studied over the last three decades following their introduction in [1] and [2] by S. Eswara Rao, R.V. Moody and T. Yokonuma. Irreducible, integrable representations for the Toroidal algebras with finite dimensional weight spaces were classified by S. Eswara Rao in [3], and that for the full toroidal Lie algebras were classified by S. Eswara Rao and Cuipo Jiang in [4]. 
\par Given a Lie algebra $\mathfrak{g}$ and $Z$ an affine scheme of finite type, the Lie algebra of all the regular maps from $Z$ to $\mathfrak{g}$ is called the map algebra associated with $Z$ and $\mathfrak{g}$ and is isomorphic to the Lie algebra $\mathfrak{g}\otimes B$, where $B=\mathcal{O}_{Z}(Z)$. In the last decade, the representation theory of Map Lie algebras  has emerged as an interesting object of research. Irreducible, finite dimensional representations of equivariant map algebras were studied by E. Neher, A. Savage, P. Senesi in [7]. A. Savage classified the irreducible quasifinite modules over map Virasoro algebras in [16]. Classification of highest weight integrable modules for $\mathfrak{g}_{aff} \otimes B $ was done by S. Eswara Rao and P. Batra [8], where $B$ is any finitely generated commutative associative unital algebra. S. Eswara Rao studied the irreducible representations of loop affine-Virosoro algebras in [17].  A paritial classification result for loop-Witt algebras was obtained by P. Chakraborty and S. Eswara Rao in [11].    Irreducible representation for map extended Witt algebra with finite dimensional weight spaces were classified by S. Sharma, R.K. Pandey, P. Chakraborty, and S. Eswara Rao, see [9]. In [14], the classification for irreducible integrable representation of loop Toroidal Lie algebras, specifically $\hat{\mathcal L}(\mathfrak{\dot g})(B)=\hat{\mathcal{L}}(\mathfrak{\dot g})\otimes B$, with finite dimensional weight spaces and where $B$ is any finitely generated commutative associative unital algebra was done by P. Chakraborty and P. Batra.  \par 
In this paper we deal with irreducible, integrable modules for the Map full Toroidal Lie algebra $ \tau(B)=\tau \otimes B$, with finite dimensional weight spaces, where $ B$ is any finitely generated commutative associative unital algebra over $\mathbb{C}$.
Now we give a brief description of this paper. The proof of our results is inspired by the papers [4], [5] and [9]. It follows that if $V$ is any irreducible integrable module with finite dimensional weight spaces for the Map full toroidal Lie algebra $ \tau(B) = \tau \otimes B$, then in case of non-zero central charges or zero central charges with $(\mathfrak{\dot g}\otimes A \otimes B)\cdot V \neq 0$, $V$ has the highest or lowest weight space $T$. In case of non-zero central charges, we prove that $T$ is an irreducible module for $( A_{\nu}\rtimes DerA_{\nu})\otimes B$ with finite dimensional weight spaces,  where $A_{\nu}=\mathbb{C}[t_{1}^{\pm 1},...,t_{\nu}^{\pm 1}]$. Such modules were shown to be the single point evaluation modules in [11], and hence are irreducible module for $A_{\nu}\rtimes DerA_{\nu}$ with finite dimensional weight spaces. It follows from [10] that $T$ is isomorphic to a Larsson-Shen module $F^{\alpha}(\Psi,a)$ as $A_{\nu}\rtimes Der A_{\nu} $-modules. Similarly, in case of zero central charges with $(\mathfrak{\dot g}\otimes A\otimes B)\cdot V\neq 0$, we have shown that $T$ is isomorphic to a Larsson-Shen module $ F^{\alpha}(\Psi,a) $(say) as $A\rtimes DerA $-modules. Finally, we prove that the entire module $V$ is a single point evaluation module, and therefore is an irreducible-integrable module for the underlying full Toroidal Lie algebra. Such modules are already classified in [4], see also [5]. For the case of zero central charges with $(\mathfrak{\dot g}\otimes A\otimes B)\cdot V=0 $, the problem of classification reduces  to the classification of irreducible $\mathcal{D}\otimes B$-modules with finite dimensional weight spaces. Such modules were shown to be single point evaluation modules in [9], and hence they are irreducible $\mathcal{D}$-modules with finite dimensional weight spaces. Moreover, all the irreducible $\mathcal{D}$-modules with finite dimensional weight spaces are classified in [12], therefore in that case the modules are those classified in [12].

	\section	{ Basic results and terminology.}
	Let $\mathfrak{\dot g}$ be a finite-dimensional Lie algebra over $\mathbb{C}$, let $\mathfrak{\dot{h}}$ be a cartan subalgebra of $\mathfrak{\dot g}$, $\dot\Delta$  be the set root system of $\mathfrak{\dot g}$ with respect to $\mathfrak{\dot h}$, $\dot\Delta_{+} $(resp. $\dot\Delta_{+}$) be the set of positive roots(resp. negative roots). Then $\mathfrak{\dot g} = \mathfrak{\dot h}\oplus_{\alpha \in \dot\Delta}\mathfrak{g_{\alpha}}$ is the Cartan decomposition for $\mathfrak{\dot{g}}$ with respect to the cartan $\mathfrak{\dot{h}}$. For $\alpha \in \dot{\Delta}$, let $\check\alpha \in \mathfrak{\dot h}$ be such that $\alpha(\check\alpha)=2$. We choose an $sl_2$-copy $\Span_{\mathbb{C}}\{{e_{\alpha},e_{-\alpha}},\check\alpha\}$ in $\mathfrak{\dot g}$ such that $e_{\alpha} \in \mathfrak{\dot g_{\alpha}},\; e_{-\alpha} \in \mathfrak{\dot g_{-\alpha }}$ and $[e_{\alpha}, e_{-\alpha}]=\check\alpha,\; [\check{\alpha},e_{\alpha}]=2e_{\alpha},\; [\check\alpha,e_{-\alpha}]=-2e_{-\alpha}$.\\
 Let $\mathfrak{\dot g}_{\pm} = \bigoplus_{\alpha \in {\dot\Delta_{\pm}}} \mathfrak{\dot g_{\alpha}}, \: \dot Q_{+} = \sum_{i=1}^{l} \mathbb{Z_{+}}\alpha_{i}$, where $\mathbb{Z_{+}}= \mathbb{N}\cup\{0\}$, also let $\mathbb{Z}_{-}$ denotes the set of negative integers.
	\par
	Let $A = C[t_{0}^{\pm1},...,t_{\nu}^{\pm1}](\nu\geq1)$ be the ring of Laurent polynomials in $\nu+1$\: commuting variables $t_{0},...,t_{\nu}$. For any $\underline{n}=(n_1,...,n_\nu) \in \mathbb{Z}^{\nu}$ and $ n_{0} \in \mathbb{Z},$ denote by $ t_{0}^{n_{0}}t^{\underline{n}} $ the monomial $t_{0}^{n_{0}}t_{1}^{n_{1}}\ldots t_{\nu}^{n_{\nu}}$. Then the tensor product of $\mathfrak{\dot{g}}$ and $A$, $\mathfrak{\dot g}\otimes A$ has a natural Lie algebra structure given by
	\begin{equation}
		 [x_1\otimes f_1,x_2\otimes f_2]\notag\\= [x_1,x_2]\otimes f_{1}f_{2},
	\end{equation}
where $x_1,x_2\in \mathfrak{\dot{g}},\:  f_{1},f_{2}\in A$. It is well known that the universal central extension of $\mathfrak{\dot g}\otimes A$ is the toroidal Lie algebra $ \mathcal{\tilde L}(\mathfrak{\dot g}) = \mathfrak{\dot g}\otimes A \oplus \mathcal{K}$, where $\mathcal{K}$ is spanned by the elements $\{t_{0}^{r_{0}}t^{\underline r}k_{p}\mid p=0,1,...,\nu,\: r_{0}\in \mathbb{Z},\: \underline{r}\in \mathbb{Z}^{\nu} \}$ with the relations \par
$ \sum_{p=0}^{\nu}r_{p}t_{0}^{r_{0}}t^{\underline r}k_{p}=0 $, and the bracket on $  \mathcal{\tilde L}(\mathfrak{\dot g})  $ is given by: \begin{equation}
	[g_{1}\otimes f_{1}, g_{2}\otimes f_{2}]\notag\\=[g_1,g_2]\otimes f_{1}f_{2}+ (g_{1}| g_{2})\sum_{p=0}^{\nu}(d_p(f_{1})f_{2})k_{p}, 
\end{equation}
with $\mathcal{K}$ central in $ \mathcal{\tilde L}(\mathfrak{\dot g}) $, i.e. $[\hat\tau,\mathcal{K}]=0$ and where $(\cdot|\cdot)$ is the invariant symmetric bilinear form on $\mathfrak{\dot g}$, $d_{p}$ is the pth degree derivation of $A$, i.e. 
\begin{equation*} d_{p}= t_{p}\pdv{}{t_{p}}, \; \; p=0,1,...,\nu. \end{equation*}
Let $\mathcal{D}$ be the Lie algebra of derivations on $A$, then $\mathcal{D}$ acts naturally on $\mathfrak{\dot g}\otimes A$ as \begin{equation}
D(x\otimes f)\notag=x\otimes Df,\;\; D \in \mathcal{D}, \: x \in \mathfrak{\dot g},f\in A, \end{equation} 
and $D$ can be uniquely extended to the universal central extension $  \mathcal{\tilde L}(\mathfrak{\dot g})  $ of $\mathfrak{\dot g}\otimes A $ by \begin{equation}
	t_{0}^{m_0}t^{\underline m}d_{a}(t_{0}^{n_0}t^{\underline n}k_b)\notag= n_{a}t_{0}^{m_{0}+n_{0}}t^{\underline m + \underline n}k_b + \delta_{ab}\sum_{p=0}^{\nu}m_{p}t_{0}^{m_{0}+n_{0}}t^{\underline{m}+\underline{n}}k_{p}.
	\end{equation}
The algebra $\mathcal{D}$ admits two non-trivial 2-cocycles with values in $\mathcal{K}$ (see [6]):
\begin{equation}
	\phi_1(t_{0}^{m_0}t^{\underline m}d_{a},t_{0}^{n_0}t^{\underline n}d_{b})\notag=-n_{a}m_{b}\sum_{p=0}^{\nu}m_{p}t^{m_{0}+n_{0}}t^{\underline{m}+\underline{n}}k_{p},
\end{equation} 
\begin{equation}
 \phi_{2}(t_{0}^{m_{0}}t^{\underline{m}}d_{a},t_{0}^{n_0}t^{\underline n}d_{b})\notag =m_{a}n_{b}\sum_{p=0}^{\nu}m_{p}t_{0}^{m_{0}+n_{0}}t^{\underline{m}+\underline{n}}k_{p}.
\end{equation}
Let $\phi$ be an arbitrary linear combination of $\phi_{1}$ and $\phi_{2}$. Then the Lie algebra $\tau = \mathfrak{\dot{g}}\otimes A \oplus \mathcal{K}\oplus \mathcal{D}$ is a Lie algebra with the Lie bracket: 
\begin{align*}
    \begin{split}
[t_{0}^{m_{0}}t^{\underline{m}}d_{a},t_{0}^{n_{0}}t^{\underline n}k_{b}] &=n_{a}t_{0}^{m_{0}+n_{0}}t^{\underline{m}+\underline{n}}k_{b}+ \delta_{ab}\sum_{p=0}^{\nu}m_{p}t_{0}^{m_{0}+n_{0}}t^{\underline{m}+\underline{n}}k_{p}, \\
	[t_{0}^{m_0}t^{\underline m}d_{a},t_{0}^{n_0}t^{\underline n}d_{b}] &=n_{a}t_{0}^{m_{0}+n_{0}}t^{\underline{m}+\underline{n}}d_{b}-m_{b}t_{0}^{m_0 + n_0}t^{\underline{m}+\underline{n}}d_{a} + \phi(t_{0}^{m_0}t^{\underline{m}}d_{a},t_{0}^{n_{0}}t^{\underline{n}}d_{b}), \\
[t_{0}^{m_0}t^{\underline m}d_{a},x\otimes t_{0}^{ n_{0}}t^{\underline n}] &=n_{a}x\otimes t_{0}^{m_0}t^{\underline m}.
\end{split}
\end{align*} 
The algebra $\tau$ is called the full toroidal Lie algebra associated to $\mathfrak{\dot g}$ and $\phi$. Let 
\begin{equation*}
	\mathfrak{h}= \mathfrak{\dot h}\oplus(\bigoplus_{i=0}^{\nu}\mathbb{C}k_{i})\oplus(\bigoplus_{i=0}^{\nu}\mathbb{C}d_{i}).
\end{equation*}
Then $\mathfrak{h}$ is a Cartan subalgebra of $\tau$. Let $\delta_{i},\Lambda_{i}\in \mathfrak{h}^{*}(i=0,1,...,\nu)$ be such that 
\begin{equation*}
	\Lambda_{i}(\mathfrak{\dot h})=0, \;    \Lambda_{i}(k_{j})=\delta_{ij},\;\Lambda_{i}(d_j)=0, \; i,j=0,1,...,\nu,
\end{equation*}
\begin{equation*}
	\delta_{i}(\mathfrak{\dot h})=0,\; \delta_{i}(k_j)=0, \; \delta_{i}(d_j)=\delta_{ij},\; i,j=0,1,...,\nu.
	\end{equation*}
and for $\underline{m}=(m_{1},...,m_{\nu})\in \mathbb{Z}^{\nu}$, let $\delta_{\underline m}$  denote the isotropic root $\sum_{i=1}^{\nu}m_{i}\delta_{i}$, then $\tau$ has the root space decomposition with respect to $ \mathfrak{h} $ as follows:
\begin{equation}
	\tau\notag =\mathfrak{h}\oplus(\bigoplus_{\beta \in \Delta}\tau_{\beta}),
\end{equation}
where $\Delta=\dot\Delta\cup\{\alpha+m_{0}\delta_{0}+\delta_{\underline m}\mid \alpha \in \dot\Delta\cup\{0\},\: \underline{m}\in \mathbb{Z}^{\nu},\: m_{0}\in \mathbb{Z},\: (m_{0},\underline{m})\neq (0,\underline0)\}$ and  
\begin{equation}
	\tau_{\alpha+m_{0}\delta_0 + \delta_{\underline m}}=\notag \mathfrak{\dot g_{\alpha}}\otimes t_{0}^{m_{0}}t^{\underline m},
\end{equation}
\begin{equation}
		\tau_{m_{0}\delta_{0}+\delta_{\underline m}}\notag=\mathfrak{\dot h}\otimes t_{0}^{m_0}t^{\underline m}\oplus(\bigoplus_{i=0}^{\nu}\mathbb{C}t_{0}^{m_{0}}t^{\underline m}k_{i})\oplus (\bigoplus_{i=0}^{\nu}\mathbb{C}t_{0}^{m_{0}}t^{\underline m}d_{i}).
\end{equation}
Let $ \mathfrak	{b}= \mathcal{K}\oplus \mathcal{D}$,
    then
\begin{align*}
\begin{split}    
	\mathfrak{b_{+}}&=\notag\sum_{i=0}^{\nu}t_{0}\mathbb{C}[t_{0},t_{1}^{\pm 1},...,t_{\nu}^{\pm 1}]k_{i}\oplus \sum_{i=0}^{\nu}t_{0}\mathbb{C}[t_{0},t_{1}^{\pm 1},...,t_{\nu}^{\pm 1}]d_{i}, \\
    \mathfrak{b_{-}}&=\notag\sum_{i=0}^{\nu}t_{0}^{-1}\mathbb{C}[t_{0}^{-1},t_{1}^{\pm 1},...,t_{\nu}^{\pm 1}]k_{i}\oplus \sum_{i=0}^{\nu}t_{0}^{-1}\mathbb{C}[t_{0}^{-1},t_{1}^{\pm 1},...,t_{\nu}^{\pm 1}]d_{i}, \\
    \mathfrak{b}_{0}&=\notag\sum_{i=0}^{\nu}\mathbb{C}[t_{1}^{\pm 1},...,t_{\nu}^{\pm 1}]k_{i}\oplus \sum_{i=0}^{\nu}\mathbb{C}[t_{1}^{\pm 1},...,t_{\nu}^{\pm 1}]d_{i}.
\end{split}
\end{align*}
and 
\begin{align*}
\begin{split} 
\tau_{+} &= \mathfrak{\dot g}_{+}\otimes \mathbb{C}[t_{1}^{\pm 1},...,t_{\nu}^{\pm 1}]\oplus \mathfrak{\dot g}\otimes t_{0}\mathbb{C}[t_{0},t_{1}^{\pm 1},...,t_{\nu}^{\pm 1}]\oplus \mathfrak{b}_{+}, \\
\tau_{-} &= \mathfrak{\dot g}_{-}\otimes \mathbb{C}[t_{1}^{\pm 1},...,t_{\nu}^{\pm 1}]\oplus \mathfrak{\dot g}\otimes t_{0}^{-1}\mathbb{C}[t_{0}^{-1},t_{1}^{\pm 1},...,t_{\nu}^{\pm 1}]\oplus \mathfrak{b}_{-}, \\
\tau_{0} &=\mathfrak{\dot h}\otimes \mathbb{C}[t_{1}^{\pm 1},..., t_{\nu}^{\pm 1}]\oplus \mathfrak{b}_{0}.
\end{split}
\end{align*}
Therefore
\begin{equation*}
	\mathfrak{b} = \mathfrak{b}_{+}\oplus\mathfrak{b}_{0}\oplus \mathfrak{b}_{-},
\end{equation*}
\begin{equation*}
	\tau = \tau_{+}\oplus\tau_{0}\oplus \tau_{-}.
\end{equation*}
Extend $\alpha \in\dot\Delta$ to $\mathfrak{h}^{*}$ by setting $\alpha(k_i)=\alpha(d_{i})=0$, the normal non-degenrate symmetric bilinear form $(\cdot|\cdot)$ on $\mathfrak{\dot h}^{*}$ extends to a non-degenrate symmetric bilinear form on $\mathfrak{h}^{*}$ given by
\begin{equation*}
	(\alpha_{i}|\delta_{k})\notag=(\alpha_{i}|\Lambda_{k})=0, \: 1\leq i \leq l,\:  0\leq k \leq \nu,
\end{equation*}
\begin{equation*}
	(\delta_{k}|\delta_{p})\notag=(\Lambda_{k}|\Lambda_{p})=0, (\delta_{k}|\Lambda_{p})=\delta_{kp},\: 0\leq k,p \leq \nu.
\end{equation*}
For $\gamma=\alpha+m_{0}\delta_{0}+\delta_{\underline m} \in \Delta$, where $\alpha \in \dot\Delta$, $\gamma$ is called real root, if $(\gamma|\gamma)\neq0$. Denote by $\Delta^{re}$ the set of all real roots. Define the co-root of $\gamma$ as
\begin{equation*}
	\check\gamma=\check\alpha +  \frac{2}{(\alpha|\alpha)}\sum_{i=1}^{\nu}m_{i}k_{i}. 
\end{equation*}
Then \begin{equation}
	\gamma(\check\gamma)\notag=\alpha(\check{\alpha})=2.
\end{equation}

Let $\gamma$ be a real root. Define reflection on $\mathfrak{h}^{*}$ by
\begin{equation}
	r_{\gamma}(\lambda)\notag= \lambda - \lambda(\check\gamma)\gamma,\: \lambda \in \mathfrak{h}^{*}.
\end{equation}
Let $\mathcal{W}$ be the Weyl group generated by $\{r_{\gamma}\mid \gamma \in \Delta^{re}\}$. Then the bilinear form $(\cdot|\cdot)$ is $\mathcal{W}$-invariant.

\section { The Map full toroidal Lie algebra}\label{NW}

Let $B$ be any finitely generated associative commutative unital algebra over $\mathbb{C}$, then the tensor product of $\tau $ and $B$, denoted $\tau\otimes B$, is a Lie algebra with the bracket given by:\begin{equation*}
[x_{1}\otimes b_{1},x_{2}\otimes b_{2}] = [x_{1},x_{2}]\otimes b_{1}b_{2},\; x_{1},x_{2}\in \tau,\: b_{1},b_{2}\in B. 
\end{equation*}
We denote the Map full Toroidal Lie algebra by $\tau(B)$. Explicitly, the Map full toroidal Lie algebra is
\begin{equation*}
	\tau(B)\notag=\tau\otimes B = (\mathfrak{\dot g}\otimes A\oplus\mathcal{K}\oplus\mathcal{D})\otimes B = (\mathfrak{\dot g}\otimes A \otimes B) \oplus (\mathcal{K}\otimes B) \oplus (\mathcal{D}\otimes B).
\end{equation*}
The subalgebra $\mathfrak{h}\otimes 1$ is an abelian subalgebra for $\tau(B)$ and plays the role of cartan subalgebra. Similar to the root space decomposition for $\tau$, we have the root space(w.r.t $\mathfrak{h}$) for $\tau(B)$, given by
\begin{equation}
	\tau(B)\notag =(\mathfrak{h}\oplus\bigoplus_{\beta \in \Delta}\tau_{\beta})\otimes B,
\end{equation}
where $\Delta=\dot\Delta\cup\{\alpha+m_{0}\delta_{0}+\delta_{\underline m}\mid \alpha \in \dot\Delta\cup\{0\},\:\underline{m}\in \mathbb{Z}^{\nu},\:m_{0}\in \mathbb{Z},\:(m_{0},\underline{m})\neq (0,\underline0)\}$ and  
\begin{equation}
	\tau(B)_{\alpha+m_{0}\delta_0 + \delta_{\underline m}}=\notag \mathfrak{\dot g_{\alpha}}\otimes t_{0}^{m_{0}}t^{\underline m}\otimes B,
\end{equation}
\begin{equation}
	\tau(B)_{m_{0}\delta_{0}+\delta_{\underline m}}\notag=\mathfrak{\dot h}\otimes t_{0}^{m_0}t^{\underline m}\otimes B \oplus(\bigoplus_{i=0}^{\nu}\mathbb{C}t_{0}^{m_{0}}t^{\underline m}k_{i}\otimes B)\oplus (\bigoplus_{i=0}^{\nu}\mathbb{C}t_{0}^{m_{0}}t^{\underline m}d_{i} \otimes B).
\end{equation}
With the notations consistent with underlying full Toroidal Lie algebra $\tau$, we have
\begin{align*} \begin{split}
	\mathfrak{b}\otimes B &= (\mathfrak{b}_{+}\otimes B)\oplus(\mathfrak{b}_{0}\otimes B)\oplus (\mathfrak{b}_{-}\otimes B), \\
 \tau\otimes B &= (\tau_{+}\otimes B)\oplus(\tau_{0}\otimes B)\oplus (\tau_{-}\otimes B), \\ 
 &= \tau(B)_{+}\oplus\tau(B)_{0}\oplus \tau(B)_{-}. \end{split}
\end{align*} 
where $\tau(B)_{\pm}=\tau_{\pm}\otimes B,\: \tau(B)_{0}=\tau_{0}\otimes B$ and  $b_{+},\tau_{+}$ and the other symbols have the same meaning as in the toroidal  case.	It also follows that Weyl group $\mathcal{W} $ for $\tau(B)$ remains the same as that for $\tau$.
\\ Now we define the integrable modules for $\tau(B)$
\begin{defi} A module V of $\tau(B)$ is called integrable if 
	\begin{itemize}
		\item[(1)] V is a $\mathfrak{h}$-weight module, i.e.,
		\begin{equation}
			V\notag=\bigoplus_{\lambda \in \mathfrak h^{*} }V_{\lambda},
		\end{equation}
	where $V_{\lambda}=\{v\in V\mid h\cdot v=\lambda(h)v,\: \forall\: h\in \mathfrak{h}\}$.
	\item[(2)] For $\alpha \in \dot\Delta,\: m_{0}\in \mathbb{Z},\:\underline{m}\in \mathbb{Z}^{\nu},\: b\in B$, $x_{\alpha}\otimes t_{0}^{m_{0}}t^{\underline m}\otimes b$ acts locally nilpotently on V, where $x_{\alpha}\in \mathfrak{\dot g}_{\alpha}$.
	\end{itemize} \end{defi}
Let $P(V)=\{\lambda \in \mathfrak{h}^{*}\mid V_{\lambda}\neq 0\}$ be the set of weights of $V$. Denote by $\mathcal{C}_{fin}$  the catgory of integrable $\tau(B)$-modules with finite dimensional weight spaces. Since the definition of integrable module for $\tau(B)$ implies that it is also an integrable module for $\tau$, we have

\begin{lemm}  Let $V\in \mathcal{C}_{fin}$. Then
	\begin{itemize}
		\item[(1)] $P(V)$ is $\mathcal{W}$-invariant.
		\item[(2)] $dim V_{\lambda} = dim V_{\omega\lambda}, \:\omega\in\mathcal{W},\: \lambda \in P(V) $.
		\item[(3)] If $\alpha\in \Delta^{re},\: \lambda\in P(V)$, then $\lambda(\check\alpha)\in \mathbb{Z}$.
		\item[(4)] If $\alpha\in \Delta^{re},\: \lambda \in P(V) $ and $ \lambda(\check\alpha)>0$, then $\lambda-\alpha\in P(V)$.
		\item[(5)] For $\lambda\in P(V)$, $\lambda(k_{i})$ is a constant integer, $i=0,1,...,\nu.$
	\end{itemize}
	\end{lemm}
 \begin{proof}
     The proof follows from that fact that $V$ is an integrable module for $\tau$ and prove is the same as that in [3].
 \end{proof}
It follows from the above lemma that \begin{equation*}
	\lambda(k_{i})=c_{i},\: i=0,1,...,\nu,\: \forall \: \lambda \in P(V), \: c_{i}\in \mathbb{Z}. 
	\end{equation*}
Consider the following affine subalgebra $\mathfrak{g}_{a}$ of $\tau(B)$:
\begin{equation}
	\mathfrak{g}_{a}\notag=(\mathfrak{\dot{g}}\otimes\mathbb{C}[t_{0}^{\pm 1}]\oplus\mathbb{C}k_{0}\oplus\mathbb{C}d_{0})\otimes 1,
\end{equation}
and
\begin{align*}
    \begin{split}
        \Delta_{a,+}&= \dot\Delta_{+}\cup\{\alpha+m_{0}\delta_{0} \mid \alpha \in \dot\Delta\cup\{0\}, \: m_{0} \in \mathbb{N} \}, \\
        \Delta_{a,-}&= \dot\Delta_{-}\cup\{\alpha+m_{0}\delta_{0} \mid \alpha \in \dot\Delta\cup\{0\}, \: -m_{0} \in \mathbb{N} \},
    \end{split}
\end{align*}
\begin{equation*}\mathfrak{g}_{a,+}=\bigoplus_{\alpha \in \Delta_{a,+}}\mathfrak{g}_{a,\alpha}, \: \: \: \: 
    \mathfrak{g}_{a,-}=\bigoplus_{\alpha \in \Delta_{a,-}}\mathfrak{g}_{a,\alpha}, \: \: \: \:   \mathfrak{h}_{a} = (\mathfrak{\dot h}\oplus \mathbb{C}k_{0}\oplus \mathbb{C}d_{0})\otimes 1,
\end{equation*}
where $\mathfrak{g}_{a,\alpha}= \{x\in \mathfrak{g}_{a}\mid [h,x]=\alpha(h)x,\: h\in \mathfrak{h}_{a}\}$. Then, it is clear that\begin{equation*}
    \mathfrak{g}_{a}=\mathfrak{g}_{a,+}\oplus \mathfrak{h}_{a}\oplus \mathfrak{g}_{a,-}.
\end{equation*}
The following lemma considers twisting the Lie algebra $\tau(B)$ by a natural automorphism.
\begin{lemm} Let $\textbf{A}=(a_{ij})(0\leq j \leq \nu)$ be a matrix of order $\nu+1$ such that $det\textbf{A}=1$ and $a_{ij}\in \mathbb{Z}$. Then there exists an automorphism $\sigma$ of $\tau(B)$ such that

	\begin{equation}
		\sigma(x\otimes t^{\textbf{m}}\otimes b)\notag=x\otimes t^{\textbf{m}\textbf{A}^{t}}\otimes b,
	\end{equation}
	\begin{equation}
		\sigma(t^{\textbf{m}}k_{j}\otimes b)\notag=\sum_{p=0}^{\nu}a_{pj}t^{\textbf{m}\textbf{A}^{t}}k_{p}\otimes b,\: 0\leq\ j\leq\nu,
	\end{equation}
	\begin{equation}
		\sigma(t^{\textbf{m}}d_{j}\otimes b)\notag=\sum_{p=0}^{\nu}b_{jp}t^{\textbf{m}\textbf{A}^{t}}d_{p}\otimes b, \: 0\leq\ j\leq\nu,
	\end{equation}
	where $(b_{ij})=\textbf{A}^{-1}$ and $b\in B  $.

 \end{lemm}
If the center acts non-trivially, then by the above lemma, we can twist the Lie algebra $\tau(B)$ such that the central charges are $c_{0}\neq0,\:c_{1}=c_{2}=...=c_{n}=0$. Following theorem gives the existence of the highest weight space(or the lowest space) in both the cases, namely, center acting non-trivially and center acting non-trivially.
\begin{theo}
Let $V\in \mathcal{C}_{fin} $. Then
\begin{itemize}
	\item[(1)] If $c_{0}>0$ and $c_{1}=c_{2}=\cdots=c_{\nu}=0$, then
	\begin{equation}
		\{v\in V\mid \tau(B)_{+}\cdot v\notag=0\}\neq 0.
	\end{equation}
\item[(2)] If $c_{0}<0$ and $c_{1}=c_{2}=\cdots=c_{\nu}=0$, then
\begin{equation}
   \{v\in V \mid \tau(B)_{-}\cdot v\notag=0\}\neq 0.
\end{equation}
\item[(3)] If $c_{0}=c_{1}=\cdots=c_{\nu}=0$, then there exists non-zero elements $ v,w\in V$
 such that \begin{equation}
 	(\mathfrak{\dot{g}}_{+}\otimes A\otimes B)\cdot v\notag=0, \; \; \text{and} \; \; 
(\mathfrak{\dot{g}_{-}}\otimes A\otimes B)\cdot w=0.
 	 \end{equation}
\end{itemize}
\end{theo}
\begin{proof} On the same lines of Theorem 2.1 of  [5], one can deduce that the set $\{\lambda \in P(V) \mid V_{\lambda + \eta } =0,\:  \forall \eta \in \dot{Q}_{+}-\{0\} \}$ is non-empty. \\
    (1) Since $V_{\underline\lambda}$ is an irreducible $\mathfrak{g}_{a}$-module with finite dimensional weight spaces, by Lemma 3.4 of [14], there exists $\lambda \in P(V_{\underline{\lambda}})$ such that $(\lambda|\alpha)\geq 0$ and \begin{equation*}
        V_{\lambda + \alpha}=0, \; \; \;  \forall \alpha \in \Delta_{a,+}.
    \end{equation*}
The remaining proof is the same as that of Theorem 2.1 of [5].

\end{proof}
\section {  Irreducible modules of $\tau(B)$ in $\mathcal{C}_{fin}$ with non-trivial action of center }
Using Lemma 3.4, we assume that $c_{0}\neq0,\: c_{1}=c_{2}=\cdots=c_{\nu}=0$ and the case $c_{0}<0$ is similar. Let \begin{equation*}
	 T=\{v\in V \mid \tau(B)_{+}\cdot v=0\}.
\end{equation*}
One can verify that $T$ is a $\tau(B)_{0}$-module. Since V is irreducible, it follows that \begin{equation*}
	V=U(\tau(B)_{-})\cdot T.
	\end{equation*}
Therefore, $T$ is irreducible as a $\tau(B)_{0}$-module. Let \begin{equation}
	T\notag=\bigoplus_{\underline m \in \mathbb{Z}^{\nu}}T_{\underline m},
\end{equation}
where $T_{\underline m}=\{v\in T\mid d_{i}\otimes 1\cdot v=(\lambda_{0}(d_{i})+m_{i})v,\: 1\leq i \leq \nu \}$, for some $\lambda_{0}\in P(V),\; \underline{m}=(m_{1},m_{2},...,m_{\nu})\in \mathbb{Z}^{\nu} $. Therefore, $T$ is $\mathbb{Z}^{\nu}$-graded and $T_{\underline{m}}$ are all finite dimensional.
 We record the following Lemmas important for our purpose.

\begin{lemm}
	If there exists a non-zero vector $w$ in $ T$ such that $t^{\underline m}k_{i}\otimes b\cdot w=0 $ for \(0\leq i \leq \nu\), then $t^{\underline m}k_{i}\otimes b $ is locally nilpotent on $ T$.
\end{lemm}
	\begin{proof} Note that for $\underline{r}\in \mathbb{Z}^{\nu},\: i,j\in \{0,1,...,\nu \},\: b\in B$, we have
	\begin{equation*}
		[[t^{\underline{r}}d_{a}\otimes b, t^{\underline m}k_{i}\otimes 1],t^{\underline m} k_{i}\otimes 1]\notag=[[t^{\underline{r}}d_{a},t^{\underline{m}}k_{i}],t^{\underline{m}}k_{i}]\otimes{b}=0.
	\end{equation*}
Therefore if there exists some  $0\neq w\in T$ such that $t^{\underline m}k_{i}.w=0$, then it can be shown inductively that 
\begin{equation*}
	(t^{\underline m}k_{i}\otimes 1)^{k+1}\cdot (t^{\underline{r_1}}d_{i_1}\otimes b_{i_1}\cdot 
t^{\underline r_2}d_{i_2}\otimes b_{i_2}\cdots t^{\underline r_k}d_{i_k}\otimes b_{i_k})\cdot w=0,
\end{equation*}

 for any $k\in \mathbb{N},\: \underline r_{1},\underline r_{2},...,\underline r_{k} \in \mathbb{Z}^{\nu},\: i_{j} \in \{0,...,\nu \},\: b_{i_j}\in B,\:\forall \: j\in \{1,...,k\}$. 
 As $T$ is irreducible $\tau(B)_{0}$-module, we get that $t^{\underline m} k_{i}\otimes 1$ is locally nilpotent. \end{proof}
 \begin{lemm}
	For any $\underline{m}\in \mathbb{Z}^{\nu}$, $v\in T\setminus\{0\}$, we have $t^{\underline m}k_{0}\otimes1\cdot v\neq 0$, and hence 
 \begin{equation*}
		dimT_{\underline m}=dimT_{\underline n}, 
\end{equation*}
for all \underline{m},$\underline{n}\in \mathbb{Z}^{\nu}$.
\end{lemm}
\begin{proof} The proof is similar to [4] and [5]. \end{proof}

By Lemma 4.2, we let 
\begin{equation}
	dimT_{\underline m}\notag=dimT_{\underline n}=n, \: \forall\: \underline{m},\underline{n}\in \mathbb{Z}^{\nu}.
\end{equation}
Let $\{v_{1},v_{2},...,v_{n}\}$ be a basis of $ T_{\underline{0}}$ and if 
\begin{equation*}
	v_{i}(\underline m)=\frac{1}{c_{0}}t^{\underline m}k_{0}\otimes1\cdot v_{i},1\leq i \leq n.	
\end{equation*}
It follows from Lemma 4.2 that $\{v_{i}(\underline m)\mid 1\leq i\leq n \}$ is a basis for $T_{\underline m}$. Assume that
\begin{equation*}
	\frac{1}{c_{0}}t^{\underline m}k_{0}\otimes 1\cdot (v_{i}(\underline n))_{i=1}^{n}=(v_{i}(\underline m +\underline n))_{i=1}^{n}\textbf B_{\underline m,\underline n}.
\end{equation*}
Hence $\textbf{B}_{\underline m,\underline n}$ is an invertible matrix of order $n$ and by the definition of $\textbf B_{\underline m, \underline n}$, we have \begin{equation}
	\textbf B_{\underline m,\underline n}\textbf B_{\underline r,\underline s}\notag= \textbf B_{\underline {r}, \underline {s}}\textbf B_{\underline m,\underline n}, \: \textbf B_{\underline m, \underline n}=\textbf B_{\underline n, \underline m}.
\end{equation}
Therefore, all the matrices $\textbf  B_{\underline m, \underline n},\underline m, \underline n \in \mathbb{Z}^{\nu}$ can be taken as upper triangular matrices.
\begin{lemm}
	For any $\underline{m},\underline{n}\in \mathbb{Z}^{\nu},\ \textbf B_{\underline m, \underline n}-\textbf I $ is a stricly upper triangular matrix.
\end{lemm}
 \begin{proof}
 	Similar to [4] and [5].
 \end{proof} 
Let
\begin{equation}\label{4.1}
	\frac{1}{c_{0}}t^{\underline m}k_{0}\otimes b\cdot (v_{i})_{i=1}^{n} = (v_{i}(\underline m))_{i=1}^{n}\textbf B_{\underline m,b},
\end{equation} 
and it is evident that $\textbf{B}_{ \underline m. \underline n}\textbf B_{\underline r, b}=\textbf{B}_{\underline{r},b}\textbf{B}_{\underline{m},\underline{n}} $ for every $\underline{m},\underline{n},\underline{r}\in \mathbb{Z}^{\nu}$ and $b\in B$.

For any $b\in B$, $k_{0}\otimes b$ act as a scalar on V, and therefore we have the following lemma.
\begin{lemm}
	There exists a linear map $\psi:B\to \mathbb{C}$ such that $k_{0}\otimes b\cdot v = c_{0}\psi(b)v$, for every $v\in T$ and $\psi({1})=1$.
\end{lemm}

\begin{proof}
Follows from the fact that $k_{0}\otimes b  \big|_{V_{\lambda}}:V_{\lambda}\to V_{\lambda} $  is central, $V_{\lambda}$'s are finite dimensional and $V$ is an irreducible module.
\end{proof}
\begin{lemm}
 There exists linear maps $\psi_{i}:B\to \mathbb{C}$ for $1\leq i \leq \nu$ such that $k_{i}\otimes b=\psi_{i}(b)$ on $T$ and $\psi_{i}(1)=0$, $1\leq i\leq \nu$.
 \end{lemm}
 Let 
\begin{equation*} t^{\underline m}d_{a}\otimes 1\cdot(v_{i}(\underline n))_{i=1}^{n}=(v_{i}(\underline m + \underline n))_{i=1}^{n}\textbf{A}_{\underline m, \underline n}^{(a)},\: \forall \: \underline m, \underline n \in \mathbb{Z}^{\nu}, \: 1\leq a \leq \nu.
\end{equation*}
Then as $[t^{\underline m}d_{a}\otimes 1, t^{\underline n}k_{0}\otimes 1 ]=n_{a}t^{\underline m + \underline n}k_{0}$, therefore we have
\begin{equation}\label{4.2}  \textbf A_{\underline m, \underline n}^{(a)}= \textbf B_{\underline m, \underline n}\textbf A_{\underline m, \underline 0}^{(a)} + n_{a}\textbf I. 
\end{equation}
\begin{lemm}
If $ b \in ker{\psi}$, then $ t^{ \underline m}k_0\otimes b $ is locally nilpotent on $T$ for every $\underline{m}\in \mathbb{Z}^{\nu}$.\end{lemm}
\begin{proof}
For $\underline m \in \mathbb{Z}^{\nu}$, we can assume that $m_{a}\neq 0 $ for some $ a\in \{1,2,...,\nu\}$. Suppose, for the sake of contradiction, that the result is not true, then by Lemmas 4.1 and 4.2, it follows that $t^{\underline m}k_{0}\otimes b$ is an invertible operator, and \eqref{4.1} implies that $\textbf B_{\underline m, b}$ is an invertible matrix and that
\begin{equation}
	[t^{-\underline{m}}d_{a}\otimes 1, t^{\underline{m}}k_0\otimes b]\notag=0,
\end{equation}
which gives
\begin{equation*}
	t^{-\underline{m}}d_{a}\otimes 1\cdot t^{\underline m}k_{0}\otimes b\cdot(v_{i})_{i=1}^{n}= t^{\underline m}k_{0}\otimes b\cdot t^{-\underline{m}}d_{a}\otimes 1\cdot(v_{i})_{i=1}^{n},
\end{equation*}
and therefore 

\begin{equation*}
	\textbf{A}_{-\underline m, \underline m}^{(a)}\textbf{B}_{\underline m, b} = \textbf{B}_{-\underline m, \underline m}\textbf{B}_{\underline m, b}\textbf{A}_{-\underline m, 0}^{(a)},
\end{equation*}
using \eqref{4.2}, we have 
\begin{equation*}
	(m_{a}\textbf{I}+ \textbf{B}_{-\underline m,\underline m}\textbf{A}_{-\underline m, \underline 0}^{(a)})\textbf{B}_{\underline m,b} = \textbf{B}_{\underline m,b}\textbf B_{-\underline m, \underline m}\textbf{A}_{-\underline m, \underline 0}^{(a)},
\end{equation*}
and hence
\begin{equation*}
	\textbf{B}_{\underline m, b}^{-1}(m_{a}\textbf{I}+\textbf{B}_{-\underline{m},\underline{m}}\textbf A_{-\underline m, \underline 0}^{(a)})\textbf{B}_{\underline m,b}=\textbf B_{-\underline m,\underline m}\textbf A_{-\underline m, 0}^{(a)},
\end{equation*}
which gives a contradiction. Hence, $t^{\underline m}k_{0}\otimes b$ is locally nilpotent for every $\underline m \in \mathbb{Z}^{\nu}, \: b\in ker\psi $, and therefore all the matrices $B_{\underline m, b}$ are strictly upper triangular.  \end{proof}
We will now prove a similar lemma for the operators $t^{\underline{m}}k_{i}\otimes b $  $    (1\leq i \leq \nu )$. Before that we need some notation as before.
\begin{equation*}
t^{\underline m}k_{p}\otimes b\cdot(v_{i})_{i=1}^{n}=(v_{i}(\underline{m}))_{i=1}^{n}C_{\underline{m},b}^{(p)},
\end{equation*}
Then
\begin{equation*} t^{\underline m}k_{p}\otimes b\cdot(v_{i}(\underline r))_{i=1}^{n} =\frac{1}{c_{0}}t^{\underline r}k_{0}\otimes1\cdot(v_{i}(\underline m ))_{i=1}^{n}C_{\underline{m},b}^{(p)}.
\end{equation*}
Therefore
\begin{align*}
    \begin{split}
        t^{\underline m}k_{p}\otimes b_{1}\cdot t^{\underline n}k_{q}\otimes b_{2}\cdot (v_{i})_{i=1}^{n}&= t^{\underline m}k_{p}\otimes b_{1}\cdot(v_{i}(\underline n))_{i=1}^{n}\textbf{C}_{\underline n, b_{2}}^{(q)}, \\
        &=t^{\underline m}k_{p}\otimes b_{1}\cdot \frac{1}{c_{0}}t^{\underline n}k_{0}\otimes 1(v_{i})_{i=1}^{n}C_{\underline n , b_{2}}^{(q)}, \\
        &= \frac{1}{c_{0}}t^{\underline m}k_{0}\otimes 1\cdot \frac{1}{c_{0}}t^{\underline n}k_{0}\otimes 1\cdot (v_{i})_{i=1}^{n}\textbf{C}_{\underline m,b_{1}}^{(p)}\textbf{C}_{\underline n, b_{2}}^{(q)}.
    \end{split}
\end{align*}
Hence
\begin{equation*}
    \textbf{C}_{\underline{m},b_{1}}^{(p)}C_{\underline{n},b_{2}}^{(q)}=C_{\underline{n},b_{2}}^{(q)} \textbf{C}_{\underline{m},b_{1}}^{(p)}.
\end{equation*}
Therefore the Lie algebra $\mathcal{N}$ spanned by the set $\{\textbf{C}_{\underline{m},b}^{(p)} \mid \underline{m}\in \mathbb{Z}^{\nu},\: b\in B \}$ is an abelian subalgebra of $gl_{\nu}$.
\begin{lemm} Each $t^{\underline{m}}k_{i}\otimes b $ is locally nilpotent on $T$, for all $b\in\bigcap_{j=1}^{\nu} \textnormal{ker}\psi_{j}, \:\underline{m}\in \mathbb{Z}^{\nu}, \: 1\leq i \leq \nu$. 
	
\end{lemm}
\begin{proof} For $ b\in \bigcap_{j=1}^{\nu}\text{ker}\psi_{j}$ and $\underline{m}\in \mathbb{Z}^{\nu}$ with $m_{a}\neq 0$ for some $a\in \{1,2,...,\nu\}$.  If $ t^{\underline{m}}k_{p}\otimes b $ is not locally nilpotent, then proceeding in a manner similar to Lemma 4.6, we get
\begin{equation*}
    (\textbf{C}_{\underline{m},b}^{(p)})^{-1}(\textbf{A}_{-\underline{m},\underline{0}}^{(a)}+ m_{a}\textbf{B}_{-\underline{m},\underline{m}}^{-1})\textbf{C}_{\underline{m},b}^{(p)}= \textbf{A}_{-\underline{m},\underline{0}}^{(a)},
\end{equation*}
which leads to a contradiction, and hence the result follows.
\end{proof}
By Lemma 4.7, for each $ i \in \{1,2,...,\nu \} $ and each $b\in \bigcap_{j=1}^{n}\text{ker}\psi_{j}$,  there is a positive integer $N\equiv N(i,b,\underline{r})$ such that $(t^{\underline m}k_{i}\otimes b)^{N}(v_{i})_{i=1}^{n}=0$. Consequently, $(t^{\underline m}k_{i}\otimes b)^{N}\cdot T_{\underline 0}=0$, and hence $(t^{\underline m}k_{i}\otimes b)^{N}\cdot T=0$. Now we prove the following important lemma. 
\begin{lemm}
$t^{\underline m}k_{i}\otimes bb'(1\leq i \leq \nu)$ is locally nilpotent for every $\underline m \in \mathbb{Z}^{\nu}$, for every $b\in \bigcap_{j=1}^{\nu}\textnormal{ker} \psi_{j} $ and $b'\in B$.\end{lemm}
\begin{proof} Consider $\underline{r}$ such that $r_{a}\neq0$ for some $a\in\{1,2,...,\nu \}$, then by Lemma 4.7, $t^{\underline r}k_{i}\otimes b$ is locally nilpotent for every $\underline r \in \mathbb{Z}$. For fixed $i\in \{1,2,...,\nu\}, \: b\in \bigcap_{j=1}^{\nu}\textnormal{ker} \psi_{j}$, if $N\equiv N(i,b,\underline{r})$ is the least positive integer such that $(t^{\underline{r}}k_{i}\otimes b)^{N}\cdot T=0$, then for $b' \in B$ and $v\in T$,  we consider 
	\begin{equation*}
        		t^{\underline{0}}d_{a}\otimes b'\cdot(t^{\underline r}k_{i}\otimes b)^{N}\cdot v=N r_{a}t^{\underline r}k_{i}\otimes bb'\cdot (t^{\underline r}k_{i}\otimes b)^{N-1}\cdot v+ (t^{\underline r}k_{i}\otimes b)^{N}t^{\underline 0}d_{a}\otimes b'\cdot v,
	\end{equation*}
which gives  $0\neq w \in T$ such that
\begin{equation*}
	t^{\underline{r}}k_{i}\otimes bb'\cdot w=0,
\end{equation*} and the result follows by Lemma 4.1. For the case $\underline{r}=\underline{0}$, we choose $\underline{s}\in \mathbb{Z}^{\nu}, \: a,i\in \{1,2,...,\nu\}$ $(\text{for}\; \nu\geq 2)$, such that $s_{i}\neq 0$ and $a\neq i$, then for $v\in T$
\begin{equation*}
t^{-\underline{s}}d_{a}\otimes b'\cdot(t^{\underline s}k_{i}\otimes b)^{N}\cdot v= Ns_{a}k_{i}\otimes bb'\cdot(t^{\underline{s}}k_{i}\otimes b)^{N-1}\cdot v + (t^{\underline{s}}k_{i}\otimes b)^{N}\cdot t^{-\underline{s}}d_{a}\otimes b'\cdot v,
\end{equation*}
and for $\nu=1$, if $\underline{r}\neq 0$, the operators $t_{1}^{\underline{r}}k_{1}\otimes b $ acts as zero on $T$ this follows from the fact that for $\textbf{m}=(m_{0},m_{1}) \in \mathbb{Z}^{2}$.
\begin{equation*}
    m_{0}t^{\textbf{m}}k_{0}\otimes b + m_{1}t^{\textbf{m}}k_{1}\otimes b =0
\end{equation*}
\\ For $\underline{r}=0 $, consider for $v\in T $ 
\begin{equation*}
    [t_{1}^{1}d_{0}\otimes b', t_1^{-1}k_{0}\otimes b]\cdot v= k_{1}\otimes bb'\cdot v.
\end{equation*}
For $b\in \text{ker}\psi$ , we deduce that
\begin{equation*}
    t_{1}d_{0}\otimes1\cdot(t_{0}^{-1}k_{0}\otimes b)^{N}\cdot v= Nk_{1}\otimes b\cdot(t_{0}^{-1}k_{0}\otimes b)^{N-1}\cdot v +(t_{0}^{-1}k_{0}\otimes b)^{N}t_{1}d_{0}\otimes 1\cdot v.
\end{equation*}
We thus obtain a nonzero vector $w\in T$, such that $k_{1}\otimes b\cdot w=0$. This gives that $\psi_{1}(b)=0$ and therefore $\text{ker}\psi \subset \text{ker}\psi_{1}$. Furthermore, $1\neq \text{ker}\psi$, $\psi_{1}(1)=0$ and the fact that $\text{ker}\psi$ is a subspace of $ B$ of codimension 1, it follows that $\psi_{1}\equiv 0$ on $T$.
\end{proof}
By the above Lemma, $k_{i}\otimes bb^{'}$ is locally nilpotent on $T$, and so $\psi_{i}(bb^{'})=0, 1\leq i\leq \nu$. Since $\psi_{i}=0$ for all $1\leq i \leq \nu$, it follows that $\bigcap_{i=1}^{\nu}\text{ker}\psi_{i}=B$.
\begin{lemm}For every $\underline{m}\in\mathbb{Z}^\nu,\: b\in B$, and $1\leq i\leq \nu$, the operators $t^{\underline r}k_{i}\otimes b$ act trivially on T.
\end{lemm}
\begin{proof} As a consequence of Lemma 4.8, each $ t^{\underline m}k_{i}\otimes b,\: 1\leq i\leq \nu,\: b\in B$ is locally nilpotent on $ T $. Let $\mathfrak{L}$ be the Lie subalgebra of $\tau(B)_{0}$ generated by $\{t^{\underline r}k_{i}\otimes b \mid \underline{r}\in \mathbb{Z}^{\nu}, 1\leq i\leq \nu,\: b\in B \}$. Then $W=\{v\in T \mid \mathfrak{L}.v=0\}$ is a $\tau(B)_{0}$-submodule of $T$, as for $v\in W$
	\begin{equation*}		t^{\underline s}d_{a}\otimes b^{'}\cdot t^{\underline r}k_{i}\otimes b\cdot v=t^{\underline r}k_{i}\otimes b\cdot t^{\underline s}d_{a}\otimes b^{'}\cdot v + r_{a}t^{\underline r+ \underline s}k_{i}\otimes bb^{'}\cdot v + \delta_{ai}\sum_{i=1}^{\nu}s_{j}t^{\underline r + \underline s}k_{j}\otimes bb^{'}\cdot v.	\end{equation*}
	
	  As all $C_{\underline m,b}^{(i)}$ are all nilpotent matrices and $\mathcal{N}$ is an abelian subalgebra of $gl_{\nu}$, therefore there exist $0\neq v\in T$ such that 
  \begin{equation*}
  t^{\underline{m}}k_{i}\otimes b\cdot v=0, \;  \; \forall b\in B, \; \forall \underline{m}\in \mathbb{Z}^{\nu}, 1\leq i \leq \nu. \end{equation*}
  Therefore W is a non-zero submodule of $T$ as $\tau(B)_{0}$-modules, thus $W=T$, which follows from the irreducibility of $T$.
\end{proof}
We now prove the following lemma concerning the operator $t^{\underline r}k_{0}\otimes b$, $b\in \text{ker}\psi$.
\begin{lemm} For $\underline{r}\in \mathbb{Z}^{\nu},\: b\in  \textnormal{ker}\psi$, $t^{\underline r}k_{0}\otimes b$ acts trivially on T.
\end{lemm}
\begin{proof}  Denote by $\mathfrak{L}'$ the lie subalgebra of $\tau(B)_{0}$ generated by $\{ t^{\underline r}k_{0}\otimes b\mid \underline{r}\in\mathbb{Z}^{\nu},\: b\in ker\psi\}$, then $\mathfrak{L'}\cdot T$ is $\tau(B)_{0}$-submodule of $T$ as \begin{equation*}
		t^{\underline r}d_{a}\otimes b'.t^{\underline s}k_{0}\otimes b\cdot v=s_{a}t^{\underline r + \underline s}k_{0}\otimes bb'\cdot v + \delta_{a0}\sum_{p=1}^{\nu}r_{p}t^{\underline r + \underline s}k_{p}\otimes bb'\cdot v,
	\end{equation*}
By Lemma 4.9, we obtain
\begin{equation*}
	t^{\underline r}d_{a}\otimes b'\cdot t^{\underline s}k_{0}\otimes b\cdot v=s_{a}t^{\underline r + \underline s}k_{0}\otimes bb' \cdot v.
\end{equation*}
 Since all $\textbf B_{\underline n,b}, \: b\in \textnormal{ker}\psi$, are nilpotent matrices, by the similar reasoning as in the proof of Lemma 4.9, it follows that $\{v\in T \mid \mathfrak{L}'.v=0\}$ is a proper submodule of $T$, and therefore $\mathfrak{L}'\cdot T=0$ and the result follows.
\end{proof}
Note that for any $b\in B, \: b-\psi(b)\cdot 1\in \text{ker}\psi$, and therefore by Lemma 4.10, it follows that 
\begin{equation}\label{4.3}
	t^{\underline m}k_{0}\otimes b\cdot v=\psi(b)t^{\underline{m}}k_{0}\otimes 1\cdot v, \:\text{for every}  \: b\in B,\: \underline{m}\in \mathbb{Z}^{\nu},\: v\in T.
\end{equation}
Now as $ker\psi$ is an ideal of codimension 1, there exists an  algebra homomorphism $\eta:B\to \mathbb{C}$ such that $\eta(1)=1$ and $\eta = c\psi$, as $1=\eta(1)=\eta(1)\eta(1)$ and $\psi(1)=1$, it follows that $c=1$ and $\psi$ is an algebra homomorphism. \\
The following theorem is regarding the associative action of operators $t^{\underline m}k_{0}\otimes b$ on $T$.
\begin{theo}For $\underline{m},\underline{n}\in \mathbb{Z}^{\nu},\: b_{1},b_{2}\in B,\:w\in T $, we have 
	\begin{equation}
	t^{\underline m}k_{0}\otimes b_{1}\cdot t^{\underline n}k_{0}\otimes b_{2}\cdot w\notag = c_{0}t^{\underline m +\underline n}k_{0}\otimes b_{1}b_{2}\cdot w.\end{equation}
\end{theo}
\begin{proof}
For \(\underline{m},\underline{n}\in \mathbb{Z}^{\nu},\: b_{1},b_{2}\in B\), consider the expression 
\begin{equation}
		(t^{\underline m}k_{0}\otimes b_{1}.t^{\underline n}k_{0}\otimes b_{2}\notag - c_{0}t^{\underline m +\underline n}k_{0}\otimes b_{1}b_{2}).(v_{i})_{i=1}^{n},
	\end{equation}
\begin{align*} \begin{split} 
        &=(v_{i}(\underline{m}+\underline{n}))_{i=1}^{n}(\psi(b_{1})\psi(b_{2})t^{\underline{m}}k_{0}\otimes{1}\cdot t^{\underline{n}}k_{0}\otimes 1-c_{0}\psi(b_{1}b_{2})t^{\underline{m}+\underline{n}}k_{0}\otimes 1),\\
	&=(v_{i}(\underline{m}+\underline{n}))_{i=1}^{n}c_{0}^{2}(\psi(b_{1})\psi(b_{2})\textbf{B}_{\underline{m},\underline{n}}-\psi(b_{1}b_{2})\textbf{I}), \\
 & =(v_{i}(\underline{m}+\underline{n}))_{i=1}^{n}\psi(b_{1}b_{2})c_{0}^{2}(\textbf{B}_{\underline{m},\underline{n}}-\textbf{I}).
\end{split}
	\end{align*}
Therefore, the operators $(t^{\underline m}k_{0}\otimes b_{1}.t^{\underline n}k_{0}\otimes b_{2}-c_{0}t^{\underline m + \underline n}k_{0}\otimes b_{1}b_{2})$ are locally nilpotent for all $\underline{m},\underline{n}\in \mathbb{Z}^{\nu},\:b_{1},b_{2}\in B$. Let $\mathfrak{Z}$ be the Lie-subalgebra of $\tau(B)_{0}$ generated by all such operators, then
$ \{ v\in T \mid \mathfrak{Z}.v=0 \} $ is a non-zero submodule of $T$ and  irreducibility of $T$ gives

\begin{equation*}
	\mathfrak{Z}.T=0
\end{equation*}
which completes the proof.
\end{proof}
Let $S=\Span_{\mathbb{C}}\{\mathfrak{h},k_{0},d_{0}\}$,\: then $S$ is an abelian Lie algebra and, for every $\underline{m}\in \mathbb{Z}^{\nu}$, we identify the elements $t^{\underline{m}}k_{0}\otimes1$ with $k_{0}\otimes t^{\underline{m}}\otimes 1$ and  the elements  $t^{\underline{m}}d_{0}\otimes 1$ with $d_{0}\otimes t^{\underline m}\otimes 1$. Then $T$ is uniformly bounded, irreducible $(S\rtimes DerA_{\nu})\otimes B$-module and as $S$ acts non-trivially on $T$, we have following Proposition of Sachin Sharma et. al [9]. 
\begin{prop} \textbf{[9, Proposition 3.11]} There exists a subspace $\dot S$ of S of codimension 1 such that $\dot{S}\otimes A_{\nu}\otimes B$ acts trivially on T, and hence T is an irreducible  $(A_{\nu}\rtimes DerA_{\nu})\otimes B$-module. 
\end{prop}
\begin{proof} In the present context this theorem can also be proved in the following way:
Note that for any $1\leq i\leq l $, the elements $c_{0}\check\alpha_{i}-\Lambda(\check\alpha_{i})k_{0}$ acts as zero on $T$, and on proceeding similar to Lemma 4.6 and 4.10, we obtain
\begin{equation}\label{4.4}
c_{0}\check\alpha_{i}\otimes t^{\underline m}\otimes b\cdot v = \Lambda(\check\alpha)t^{\underline m}k_{0}\otimes b \cdot v \; \; \forall \: 1\leq i\leq l,\: \underline{m}\in \mathbb{Z}^{\nu},\: b\in B,\: v\in T,
\end{equation}
and 
\begin{equation}\label{4.5}
c_{0}t^{\underline{m}}d_{0}\otimes b = \Lambda(d_{0})t^{\underline{m}}k_{0}\otimes b.
\end{equation}
    Consider the basis $\{c_{0}\check{\alpha_{i}}-\Lambda(\check\alpha_{i})k_{0},c_{0}d_{0}-\Lambda(d_{0})k_{0},k_{0} \mid 1\leq i \leq l \}$ of $S$, then using \eqref{4.4} and \eqref{4.5}, the result follows for  $\dot{S}=\{c_{0}\check{\alpha_{i}}-\Lambda(\check\alpha_{i})k_{0},\: c_{0}d_{0}-\Lambda(d_{0})k_{0} \mid 1\leq i \leq l \}$.
\end{proof}
 Using \eqref{4.3} and \eqref{4.4}, we have
 \begin{equation}
     \check\alpha_{i}\otimes t^{\underline{m}}\otimes b = \psi(b)\check\alpha_{i}\otimes t^{\underline m}\otimes 1 \; \; \text{on $T$},1\leq i\leq l.
 \end{equation}
 Similarly,
 \begin{equation}
     t^{\underline{m}}d_{0}\otimes b = \psi(b)t^{\underline m}d_{0}\otimes 1 \; \; \text{on $T$}.
 \end{equation}
 With the above chosen basis, the element $k_{0}\otimes 1 $ acts non-trivially on $T$ and by Theorem 4.11, the action of the operators $\frac{1}{c_{0}}t^{\underline m}k_{0}\otimes b$ is associative on $T$, therefore we have the following theorem from P. Chakraborty and S. Eswara Rao [11].
    \begin{theo}
    Let V be an irreducible uniformly bounded module for $(DerA_{\nu}\ltimes A_{\nu})\otimes B$. If $\frac{1}{c_{0}}t^{0}k_{0}\otimes 1$ acts as an identity on V, then V is a single point evaluation module.
\end{theo} 
Therefore, we have 
\begin{equation}
t^{\underline m}d_{p}\otimes b = \psi(b)t^{\underline m}d_{p}\otimes 1, \; \forall \: \underline{m}\in \mathbb{Z}^{\nu},\: b\in B,\: 1\leq p \leq \nu. 
\end{equation}
 We now prove that such relation also holds for the action of any homogeneous $\mathbb{Z}^{\nu+1}$-graded element of $\tau(B)_{-}$ on $T$.
\\ The $\{0\}\times \mathbb{Z}^{\nu}$-graded elements of $\tau(B)_{-}$ like $x_{-\alpha}\otimes t^{\underline{m}}\otimes b$ (for $\alpha \in \dot\Delta_{+})$ acts as $\psi(b)x_{-\alpha}\otimes t^{\underline {m}}\otimes 1$ as for any $v\in T$, the vector $x_{-\alpha}\otimes{t^{\underline m}\otimes b}\cdot v - \psi(b)x_{-\alpha}\otimes t^{\underline m}\otimes 1\cdot v$ is annihilated by each positive root vector and hence by $\tau(B)_{+}$, and therefore, this vector must be zero, otherwise we get a highest weight vector of weight $\Lambda-\alpha$, a contradiction. Therefore
\begin{equation}
x_{-\alpha}\otimes t^{\underline{m}}\otimes b = \psi(b)x_{-\alpha}\otimes t^{\underline{m}}\otimes1 \;  \text{on $T$}, \: \forall \: \alpha \in \Dot\Delta_{+},\:\underline{m}\in \mathbb{Z}^{\nu}.
\end{equation}
\\ We will now demonstrate that $t_{0}^{m_{0}}t^{\underline m}d_{0}\otimes b = \psi(b)t_{0}^{m_{0}}t^{\underline{m}}d_{0}\otimes 1$, for $m_{0}<0$, and the case $m_{0}=0$ is already done above in 4.13. The proof is inductive, so it is enough to prove the result for $m_{0}=-1$.

\begin{lemm}
    $t_{0}^{-1}t^{\underline{m}}d_{0}\otimes b \cdot v= \psi(b)t_{0}^{-1}t^{\underline m}d_{0}\otimes 1 \cdot v $ for $v\in T, \:\underline{m}\in\mathbb{Z}^{\nu},\:b\in B$.
\end{lemm}
\begin{proof} For any $v\in T$, the vector $t_{0}^{-1}t^{\underline m}d_{0}\otimes b.v - \psi(b)t_{0}^{-1}t^{\underline m}d_{0}\otimes 1.v$ is annihilated by each $\{n_{0}\}\times\mathbb{Z}^{\nu}$-graded positive root vectors with $n_{0}>0$, this follows from (4.7), (4.8) and the fact that $\tau(B)_{+}\cdot T=0$. For any $\alpha\in \dot\Delta$, consider \begin{align*} \begin{split}
x_{\alpha}\otimes t^{\underline n}\otimes b'\cdot t_{0}^{-1}t^{\underline m}d_{0}\otimes b\cdot v  &= t_{0}^{-1}t^{\underline{m}}d_{0}\otimes b\cdot x_{\alpha}\otimes t^{\underline{n}}\otimes b'\cdot v, \\ 
&= 0 \; \; (\because \tau(B)_{+}\cdot T=0),
\end{split}
\end{align*}
and the result follows.
\end{proof}
Inductively, it follows from Lemma 4.14 that $t_{0}^{m_{0}}t^{\underline{m}}d_{0}\otimes b=\psi(b)t_{0}^{m_{0}}t^{\underline{m}}d_{0}\otimes 1 \; \forall \: m_{0}<0,\:\underline{m}\in \mathbb{Z}^{\nu},\:b\in B.$ 
\begin{lemm}
    For $m_{0}<0, \underline{m}\in\mathbb{Z}^{\nu},b\in B$, we have
   \begin{align*} \begin{split}
    x_{-\alpha}\otimes t_{0}^{m_{0}}t^{\underline{m}}\otimes b &= \psi(b)x_{-\alpha}\otimes t_{0}^{m_{0}}t^{\underline{m}}\otimes 1 \; \; \forall \: \alpha \in \dot\Delta, \; \text{on $T$},
    \\
    h\otimes t_{0}^{m_{0}}t^{\underline{m}}\otimes b &=\psi(b)h\otimes t_{0}^{m_{0}}t^{\underline{m}}\otimes 1 \; \; \forall \; h\in \mathfrak{\dot h}, \; \text{on $T$}.
    \end{split}
    \end{align*}
\end{lemm}
\begin{proof} Again we prove the result for $m_{0}=-1$, and for $m_{0}<0$, result will follow by induction. For $\alpha \in \dot\Delta, \: b\in B,\: v \in T$, consider
\begin{align*}
\begin{split}
    x_{-\alpha}\otimes t_{0}^{-1}t^{\underline{m}}\otimes b\cdot v &=[t_{0}^{-2}d_{0}\otimes b, x_{-\alpha}\otimes t_{0}t^{\underline{m}}\otimes 1]\cdot v,
    \\ 
    &=-x_{-\alpha}\otimes t_{0}^{-1}t^{\underline{m}}\otimes 1\cdot t_{0}^{-2}d_{0}\otimes b\cdot v,
    \\
    &= -\psi(b)x_{-\alpha}\otimes t_{0}^{-1}t^{\underline{m}}\otimes1\cdot t_{0}^{-2}d_{0}\otimes 1\cdot v,
    \\
    &= \psi(b)x_{-\alpha}\otimes t_{0}^{-1}t^{\underline{m}}\otimes b\cdot v.
    \end{split}
\end{align*}
    Assume by induction that the result is true for all integers $m_{0}<0$ such that $N_{0}\leq m_{0}$, then one verifies that for any $0\neq v\in T$, the vector $x_{-\alpha}\otimes t_{0}^{N_{0}-1}t^{\underline m}\otimes b\cdot v -\psi(b)x_{-\alpha}\otimes t_{0}^{N_{0}-1}t^{\underline{m}}\otimes 1\cdot v $ is annihilated by each positive root vector and hence by $\tau(B)_{+}$, therefore, by the argument on the highest weight $\Lambda$, it has to be zero, that is \begin{equation*}
        x_{-\alpha}\otimes t_{0}^{N_{0}-1}\otimes b\cdot v = \psi(b)x_{-\alpha}\otimes t_{0}^{N_{0}-1}\otimes t^{\underline m}\otimes 1\cdot v.
    \end{equation*}
    Therefore, the first result follows. Similarly, it can be deduced that
    \begin{equation*}
        h\otimes t_{0}^{m_{0}}t^{\underline m}\otimes b\cdot v = \psi(b)h\otimes t_{0}^{m_{0}}t^{\underline{m}}\otimes 1\cdot v.
    \end{equation*}
\end{proof}
Again by using the argument on the highest weight along with Lemmas 4.14-4.15, we have 
\begin{lemm}
For $m_{0}<0,\: \underline m\in \mathbb{Z}^{\nu},\: b\in B$, the following equalities of operators holds on $T$
\begin{align*}
    \begin{split}
        t_{0}^{m_{0}}t^{\underline m}k_{0}\otimes b &= \psi(b)t_{0}^{m_{0}}t^{\underline{m}}k_{0}\otimes 1, \\
        t_{0}^{m_{0}}t^{\underline m}d_{p}\otimes b &= \psi(b)t_{0}^{m_{0}}t^{\underline{m}}d_{p}\otimes 1 (0\leq p \leq \nu), 
        \\
        t_{0}^{m_{0}}t^{\underline{m}}k_{p}\otimes b &= 0\;(1\leq p \leq \nu).
        \end{split}
\end{align*}
\end{lemm}
\begin{remark} We know that for any $\textbf{m}\in \mathbb{Z}^{\nu+1}$ and $b\in B$, we have $\sum_{i=0}^{\nu}m_{i}t^{m}k_{i}\otimes b=0$. Therefore, if $m_{0}\neq 0$, then it follows from Lemma 4.16 that $t_{0}^{m_{0}}t^{\underline{m}}k_{0}\otimes b\cdot T =0$.\end{remark}

By Theorem 4.13, $T$ is an irreducible $DerA_{\nu}\ltimes A_{\nu}$-module with finite dimensional weight spaces, and therefore we utilise the following theorem by S. Eswara Rao [10].

\begin{theo} As $A_{\nu}\rtimes DerA_{\nu}$-module, $T$ is isomorphic to $F^{\beta}(\Psi,a)$, for some $(\beta,\Psi, c)$, where $\beta =(\beta_{1},\beta_{2},...,\beta_{\nu}) \in \mathbb{C}^{\nu},\: a\in \mathbb{C},\: F^{\beta}(\Psi,a)= V(\Psi,a)\otimes A_{\nu}$ is an irreducible $DerA_{\nu}\ltimes A_{\nu}$-module such that $V(\Psi, a)$ is an $n$-dimensional irreducible module for $gl_{\nu}(\mathbb{C})$-module, and 
\begin{equation*} 
\Psi(I)=aId_{V(\Psi,a)},
\end{equation*} \begin{equation*}
t^{\underline r}d_{p}(w\otimes t^{\underline m})=(m_{p}+\beta_{p})w\otimes t^{\underline r + \underline m} + \sum_{i=1}^{\nu} r_{i}\Psi(E_{ip})w\otimes t^{\underline r + \underline m},
\end{equation*}
\begin{equation*}
t^{\underline m}(w\otimes t^{\underline n})= w\otimes t^{\underline m + \underline n},
\end{equation*}
where $w\in V(\Psi,a)$.
\end{theo}
 Therefore $T$ is isomorphic to $F^{\beta}(\Psi,a)$, for some $(\beta,\Psi,a)$ as $A_{\nu}\rtimes DerA_{\nu}$-modules.

Let $F^{\beta}(\Psi,a)$ be an irreducible module for $DerA_{\nu}\ltimes A_{\nu}$ as defined above, and $\Lambda \in \mathfrak{h}_{a}^*$ be such that $\Lambda(k_{0})=c_{0}\neq 0$. We extend $F^{\beta}(\Psi,a)$ to a $\tau_{0}(B)$-module as follows:

\begin{equation*}
t^{\underline r}k_{0}\otimes b\cdot w\otimes t^{\underline m}== c_{0}\psi(b)w\otimes t^{\underline r + \underline m}, \; \; \; \; \;  t^{\underline r}k_{p}\otimes b\cdot w\otimes t^{\underline m}=0, \; \ 1 \leq p \leq \nu, \end{equation*} \begin{equation*}
h\otimes t^{\underline r}\otimes b\cdot w\otimes t^{\underline m}=\Lambda(h)\psi(b)w\otimes t^{\underline{r}+\underline{m}}, \; \; \; \; t^{\underline r}d_{0}\otimes b\cdot (w\otimes t^{\underline m})=\Lambda(d_{0})\psi(b)w\otimes t^{\underline{r}+ \underline{m}},
\end{equation*}
where $h \in \mathfrak{\dot h}$. Denote $F^{\beta}(\Psi,a)$ by $F^{\beta}(\Psi,a, \Lambda)$ as $\tau(B)_{0}$-module. Let $\tau(B)_{+}$ acts on $F^{\beta}(\Psi,a,\Lambda)$ by zero and then we have the induced module for $\tau(B)$: \begin{equation*}
F_{\tau(B)}^{\beta}(\Psi,a, \Lambda)=Ind_{\tau(B)_{+}+\tau(B)_{0}}^{\tau(B)}(F^{\beta}(\Psi,a,\Lambda)).
\end{equation*}
Let $W$ be the $\tau(B)$-submodule of $F_{\tau(B)}^{\beta}(\Psi,a,\Lambda)$ generated by $(t_{0}^{m_{0}}t^{\underline{m}}d_{a}\otimes{b}-\psi(b)t_{0}^{m_{0}}t^{\underline{m}}d_{a}\otimes 1)\otimes T$,\:  $(t_{0}^{m_{0}}t^{\underline{m}}k_{0}\otimes b)\otimes T$,\:  $(y_{\alpha}\otimes t_{0}^{m_{0}}t^{\underline{m}}\otimes b -\psi(b)y_{\alpha}\otimes t_{0}^{m_{0}}t^{\underline{m}}\otimes 1)\otimes T$,\: $(x_{-\alpha}\otimes t^{\underline{m}}\otimes b - \psi(b)x_{-\alpha}\otimes t^{\underline{m}}\otimes 1)\otimes T $, for every $m_{0}\in \mathbb{Z}_{-},\: \underline{m}\in \mathbb{Z}^{\nu},\:\alpha \in \dot\Delta_{+},\: b\in B$ and where $y_{\alpha}=x_{\alpha},\:x_{-\alpha},\:\text{or} \; \check\alpha$ and let 
\begin{equation*}
    \Hat{F}_{\tau(B)}^{\beta}(\Psi,a,\Lambda)=F_{\tau(B)}^{\beta}(\Psi,a,\Lambda)/W.
\end{equation*}
Using the Lemma 4.14-4.16, the following result holds in accordance with [4].
\begin{theo} Let $\hat{F}^{\beta}_{\tau(B)}(\Psi,a, \Lambda)(\Lambda(k_{0})=c_{0}\neq 0)$ be as defined above. Then,
\begin{itemize}
    \item[(1)] There is a maximal one among the submodules of $\hat{F}_{\tau(B)}^{\beta}(\Psi,a, \Lambda )$ intersecting $\hat{F}_{\tau(B)}^{\beta}(\Psi,a,\Lambda )$ trivially. We denote by $\hat{F}_{\tau(B)}^{\beta}(\Psi,a, \Lambda)'$ this maximal submodule. 
    \item[(2)] $\tilde{F}_{\tau(B)}^{\beta}(\Psi,a,\Lambda)= \hat{F}_{\tau(B)}^{\beta}(\Psi,a,\Lambda)/\hat{F}_{\tau(B)}^{\beta}(\Psi,a,\Lambda)^{'}$ is an irreducible $\tau(B)$-module and $\tilde F_{\tau}^{\beta}(\Psi,a,\Lambda)$ is integrable if and only if $\Lambda$ or $-\Lambda$ is a dominant integral weight of $\mathfrak{g}_{a}$.
    \item[(3)] Let $V \in \mathcal{C}_{fin}$ be such that $c_{0}\neq 0,\: c_{1}=c_{2}=\cdots=c_{\nu}=0$. Then $V\cong \tilde F_{\tau(B)}^{\beta}(\psi,a, \Lambda)$ for some $(\beta,\Lambda,\Psi,a)$.
    \end{itemize} \end{theo} 
    
    \section {Modules of $\tau(B)$ in $C_{fin}$ with zero central charges.}
    We will classify the irreducible objects of $C_{fin} $, where center acts trivially. By Theorem 3.4(3),
    \begin{equation*} 
    T=\{v\in V \mid (\mathfrak{\dot{g}}_{+}\otimes A\otimes B)\cdot v=0 \}\neq 0.
    \end{equation*}
    Observe that $T$ is $( \mathfrak{b}\oplus \mathfrak {\dot{h}}\otimes A)\otimes B$-module. Since $V$ is irreducible it follows that 
    \begin{equation*} 
    h\otimes 1\cdot w = \Lambda(h)w,\; \; \; \; \; \; \; \; \forall\: h \in \mathfrak{\dot h},\: w\in T,
    \end{equation*}    for some $\Lambda \in P(V)$. If $\Lambda|_{\mathfrak{\dot h}}=0$, then it follows by the weight argument that $(\dot{g}\otimes A \otimes B)\cdot V=0$, and $V$ is an irreducible $(\mathfrak{b}\otimes B)$-module. Since $(\mathfrak{\dot g}\otimes A\otimes B)\cdot V=0 $, it follows that 
    \begin{align*}
    \begin{split}
        0=[x\otimes t^{\textbf{m}}\otimes b, y\otimes t^{\textbf{n}}\otimes 1]\cdot v &= [x,y]\otimes t^{\textbf{m}+\textbf{n}}\otimes b\cdot v + (x|y)\sum_{i=0}^{\nu}m_{i}t^{\textbf{m}+\textbf{n}}k_{i}\otimes b\cdot v, \\
        &= (x|y)\sum_{i=0}^{\nu}m_{i}t^{\textbf{m}+\textbf{n}}k_{i}\otimes b\cdot v,
        \end{split}
    \end{align*}
    where $x,y\in \mathfrak{\dot g}$ are such that $(x|y)\neq 0,\:v\in V$ and $b\in B$. \\
   For any fixed $p\in \{0,1,...,\nu\}$, choose $\textbf{m}\in \mathbb{Z}^{\nu+1}$ such that $m_{p}\neq 0,\: m_{i}=0,\: i\neq p$, then for any $\textbf{n} \in \mathbb{Z}^{\nu+1}$, the above identity gives
   \begin{equation*}
       t^{\textbf{m}+\textbf{n}}k_{p}\otimes b\cdot v=0.
   \end{equation*}
   Thus $\mathcal{K}$ acts trivially on $V$ and hence $V$ is an irreducible $\mathcal{D}\otimes B$-module with finite dimensional weight spaces. It follows from Sachin S. Sharma et. al. [9]
    \begin{theo}  As $\mathcal{D}\otimes B$-module, V is either an irreducible uniformly bounded module or an irreducible highest weight module for $\mathcal{D}\otimes B$. Moreover, V is a single point evaluation module. \end{theo}
    By theorem 5.1, $V$ is an irreducible module for $\mathcal{D}$ with finite dimensional weight spaces, such modules are completely classified by Yully Billig and Vyacheslav Futorny in [12]. 
    \begin{theo}
        As a $\mathcal{D}$-module, $V$ is belongs to one of the three types of modules
        \begin{itemize}
            \item a module of tensor fields $T(U,\beta), \beta \in \mathbb{C}^{n}$ and $U$ is a simple finite dimensional $gl_{n}$-module, different from an exterior power of the natural n-dimensional $gl_{n}$-module.
            \item a submodule $d\Omega^{k}(\beta) \subseteq \Omega^{k+1}(\beta)$ for $0 \leq k < n $ and $ \beta \in \mathbb{C}^{n}$.
            \item a module of the highest weight type $L(U,\beta,\gamma)^{g}$, twisted by $g\in GL_n{(\mathbb{Z})}$, where $U$ is a simple finite dimensional $gl_{n-1}$-module, $\beta \in \mathbb{C}^{n-1}$ and $ \gamma \in \mathbb{C} $.
        \end{itemize}
    \end{theo}
    
    Now we will assume that $\Lambda|_{\mathfrak{\dot h}}\neq 0$, 
 then  there exists $h_{0}\in \mathfrak{\dot h}$ such that $\Lambda(h_{0})\neq 0$. Let
    \begin{equation*} T=\bigoplus_{\textbf m \in \mathbb{Z}^{\nu+1} }T_{\textbf{m}},   \end{equation*}
    where $T_{\textbf{m}}=\{v\in T \mid d_{p}\otimes 1\cdot v=(\Lambda(d_{p}\otimes 1)+m_{p})v,\: 0\leq p \leq \nu \}$, $\textbf{m}=(m_{0},m_{1},...,m_{\nu})\in \mathbb{Z}^{\nu+1} $.
    Similar to the proof of [4] and [5], one can prove

    \begin{lemm} For any $\textbf{m}\in \mathbb{Z}^{\nu+1},\: v\in T\setminus\{0\}  $, we have $h_{0}\otimes t^{\textbf m}\cdot v\neq 0$, and \begin{equation*}
    dimT_{\textbf{m}}=dimT_{\textbf{n}}\; \; \; \; \;  \forall \textbf{m},\textbf{n}\in \mathbb{Z}^{\nu + 1}.
    \end{equation*} \end{lemm}
  
Let $\{v_{1},v_{2},...,v_{m}\}$ be a basis of $T_{\textbf 0}$. Let
\begin{equation*} v_i(\textbf m)=\frac{1}{\Lambda(h_{0})} h_{0}\otimes t^{\textbf m}\otimes 1\cdot v_{i}, \; \; \; 1\leq i \leq m.
\end{equation*}
Then $\{v_{1}(\textbf m),v_{2}(\textbf m),...,v_{m}(\textbf m) \}$ is a basis for $T_{\textbf m}$. Therefore if 
\begin{equation*} \frac{1}{\Lambda(h_{0})}h_{0}\otimes t^{\textbf m}\otimes 1\cdot (v_{i}(\textbf n))_{i=1}^{m} = (v_{i}(\textbf{m}+\textbf{n}))_{i=1}^{m}\textbf{H}_{\textbf m, \textbf n},
\end{equation*}
then $\textbf{H}_{\textbf m, \textbf n}$ is an invertible matrix and 
\begin{equation*}
    \textbf{H}_{\textbf{m},\textbf{n}}\textbf{H}_{\textbf{r},\textbf{s}}=\textbf{H}_{\textbf{r},\textbf{s}}\textbf{H}_{\textbf{m},\textbf{n}}.
\end{equation*}
Therefore all the matrices $\{\textbf{H}_{\textbf{m},\textbf{n}} \mid \textbf{m},\textbf{n}\in \mathbb{Z}^{\nu+1}\}$ can be taken as upper triangular matrices.
\begin{lemm} For $\textbf{m}, \textbf{n} \in \mathbb{Z}^{\nu+1}$, $\textbf{H}_{\textbf{m},\textbf{n}}-\textbf{I}$ is a strictly triangular matrix.
\end{lemm}
\begin{proof}
    Same as in [4].
\end{proof}
For $\textbf{m}=(m_0,m_1,...,m_{\nu}), \: \textbf{n}=(n_{0},n_{1},...,n_{\nu})\in \mathbb{Z}^{\nu+1}$, assume that 
\begin{equation*}
t^{\textbf{m}}d_{a}\otimes1\cdot (v_{i}(\textbf{n}))_{i=1}^{m}=(v_{i}(\textbf m + \textbf n))_{i=1}^{m}\textbf{A}_{\textbf{m},\textbf{n}}^{(a)},
\end{equation*}
where $\textbf{A}_{\textbf{m},\textbf{n}}^{(a)}\in \mathbb{C}^{m\times m}$. Since $[t^{\textbf{m}}d_{a}\otimes 1,h\otimes t^{\textbf n}\otimes 1]=n_{a}h\otimes t^{\textbf{m}+\textbf{n}}\otimes 1,\: 0\le a \leq \nu$.
\begin{theo} $(t^{\textbf m}k_{p}\otimes b)\cdot T=0,$ $\forall\: \textbf{m}\in \mathbb{Z}^{\nu+1},\: p=0,1,...,\nu,\:b\in B.$
\end{theo}
\begin{proof} Similar to the proof of Lemmas 4.7-4.9.
\end{proof}
\begin{lemm}  For the above $h_{0}\in \mathfrak{\dot{h}}$, $\textbf{m}\in \mathbb{Z}^{\nu+1}$ and $b\in B$, there exists an algebra homomorphism $\psi:b\to \mathbb{C}$ such that  
\begin{equation*}
h_{0}\otimes t^{\textbf{m}}\otimes b = \psi(b)h_{0}\otimes t^{\textbf{m}}\otimes 1 \; \; \text{on T}.
\end{equation*}
\end{lemm} 
\begin{proof} Similar to the proof of Lemmas 4.6 and 4.10. 
\end{proof}
The following theorem is about the associative action of the operators $h_{0}\otimes t^{\textbf{m}}\otimes b,\:\textbf{m}\in \mathbb{Z}^{\nu+1},\:b\in B$ on $T$. The proof is similar to Theorem 4.11.
\begin{theo} For all $\textbf{m},\textbf{n}\in \mathbb{Z}^{\nu +1},\: b_{1},b_{2}\in B$, we have
\begin{equation*}
    h_{0}\otimes t^{\textbf{m}}\otimes b_{1}\cdot h_{0}\otimes t^{\textbf{n}}\otimes{b_{2}}= h_{0}\otimes t^{\textbf{m}+\textbf{n}}\otimes b_{1}b_{2} \; \; \text{on $T$}.
\end{equation*}
\end{theo}
\begin{lemm}If $ \text{dim} \mathfrak{\dot h}=l$ and $ \{h_{i} \mid 0\leq i \leq l-1\}$ is a basis for $\mathfrak{\dot h}$, then $\{h_{0},\: \Lambda(h_{i})h_{0}-\Lambda(h_{0})h_{i}\mid 1\leq i \leq l-1 \}$ is a basis for $\mathfrak{\dot h}$ such that $\Lambda(h_{i})h_{0}-\Lambda(h_{0})h_{i}$ acts trivially on $T$ and hence
\begin{equation*}
    h_{i}\otimes t^{\textbf{m}}\otimes b = \psi(b)h_{i}\otimes t^{\textbf{m}}\otimes 1,(1\leq i \leq l-1).
\end{equation*}
\end{lemm}
\begin{proof}
Proof is similar to Theorem 4.16.
\end{proof}
We now show that the negative root vectors $x_{-\alpha}\otimes t^{\textbf m}\otimes b, \: \alpha \in \dot\Delta_{+}  $ acts as $\psi(b)x_{-\alpha}\otimes{t^{\textbf m}}\otimes 1$ on $T$. Note that
for any $v\in T,\: \alpha,\beta \in \dot{\Delta}_{+}$
\begin{align*}
\begin{split}
    &x_{\beta}\otimes t^{\textbf{n}}\otimes b'\cdot (x_{-\alpha}\otimes t^{\textbf{m}}\otimes b-\psi(b)x_{-\alpha}\otimes t^{\textbf{m}}\otimes 1)\cdot v \\
    &=[x_{\beta},x_{-\alpha}]\otimes t^{\textbf{\textbf{m}+\textbf{n}}}\otimes bb'\cdot v-\psi(b)[x_{\beta},x_{-\alpha}]\otimes b'\cdot v, \; (\because \: c_{0}=c_{1}=...=c_{\nu}=0 \;  \text{and} \; v\in T), \\
    &=0.
    \end{split} 
    \end{align*}

It follows that for every $v\in T$, the vector $ x_{-\alpha}\otimes t^{\textbf n}\otimes b_{2}\cdot v-\psi(b_{1})x_{-\alpha}\otimes t^{\textbf n}\otimes 1\cdot v  $ is annihilated by $\mathfrak{g}_{+}\otimes A\otimes B$. Therefore
\begin{equation*}
x_{-\alpha}\otimes t^{\textbf m}\otimes b = \psi(b)x_{-\alpha}\otimes t^{\textbf m}\otimes 1 \; \; \; \text{on $T$}
\end{equation*}
Otherwise, we obtain a highest weight vector(non-zero) of weight $\Lambda-\alpha$, which is an absurd. Thus, we have established that
\begin{theo} For $\alpha \in \dot\Delta_{+},\: \textbf m \in \mathbb{Z}^{\nu+1}$ and $b\in B$, we have \begin{equation*}
x_{-\alpha}\otimes t^{\textbf m}\otimes b\cdot v =\psi(b)x_{-\alpha}\otimes t^{\textbf m}\otimes{1}\cdot v \; \; \; \; \forall \: v\in T. \end{equation*}
\end{theo}

By Lemma 4.6, it follows that $T$ is an irreducible $(A\rtimes DerA )\otimes{B}$-module and thus, Theorem 4.13, yields that $T$ is a single point evaluation module and, hence, an irreducible $(A\rtimes DerA)$-module. Consequently, by Theorem 4.18, $T$ is isomorphic to $F^{\beta}(\Psi,a)$, for some $(\beta,\Psi,a)$, as $(A\rtimes DerA)$-modules, where $F^{\beta}(\Psi,a)=V(\Psi,a)\otimes A$ is an irreducible $A\rtimes DerA$-module such that $V(\Psi,a)$-is an $m$-dimensional irreducible $gl_{\nu}(\mathbb{C})$-module, and 
\begin{equation}
\Psi(I)\notag=aId_{V(\Psi,a)},
\end{equation}
\begin{equation*}
    t^{\textbf r}d_{p}(v\otimes t^{\textbf m})=(m_{p}+\beta_{p})v\otimes t^{\textbf{r}+\textbf{m}}+ \sum_{i=0}^{\nu}r_{i}\Psi(E_{ip})v\otimes t^{\textbf{r}+\textbf{m}}.
\end{equation*}
Let $F^{\beta}(\Psi,a)$ be the $ A\rtimes DerA$-module as defined above, then extending this module to the $(A\rtimes DerA)\otimes B$-module as follows:
\begin{equation*} t^{\textbf{m}}d_{p}\otimes b\cdot v= \psi(b)t^{\textbf{m}}d_{p}\otimes{1}\cdot v,
\end{equation*}
and
\begin{equation*}
h_{0}\otimes t^{\textbf{m}}\otimes b\cdot v = \psi(b)h_{0}\otimes t^{\textbf{m}}\otimes 1\cdot v.
\end{equation*}
Let $\Lambda_{0}\in \mathfrak{\dot h}^{*}$ be such that $\Lambda_{0}\neq 0$, $L(\Lambda_{0},\pi)$ is an irreducible highest weight module of $\mathfrak{\dot{g}}$ with the highest weight $\Lambda_{0}$ and the associated highest weight vector $v_{\Lambda_{0}}$, $F^{\beta}(\Psi,a)$ be as defined above. Let
\begin{equation*} F^{\beta}(\Psi,a,\Lambda_{0}) = V_{\Lambda_{0}}\otimes F^{\beta}(\Psi,a).
\end{equation*}
We define the action of $\tau(B)$ on $F^{\beta}(\Psi,a, \Lambda_{0})$ by
\begin{align*}
\begin{split}
    x\otimes t^{\textbf{r}}\otimes b\cdot(w\otimes v(\textbf{m})) &= \psi(b)(\pi(x)w)\otimes v(\textbf{m}+\textbf{r}), \\
t^{r}k_{p}\otimes b\cdot(w\otimes v(\textbf{m})) &= 0, \\
t^{\textbf{r}}d_{p}\otimes{b}\cdot ( w\otimes v(\textbf{m} )) &= \psi(b)w\otimes (t^{\textbf r}d_{p}\otimes 1)\cdot v(\textbf{m}),
\end{split}
\end{align*}
where $x\in \mathfrak{\dot g},\: w\in L(\Lambda_{0}),\: v(\textbf{m})=v\otimes t^{\textbf{m}}\in V(\Psi,a)\otimes A,\: b\in B,\: \textbf{r},\textbf{m}\in \mathbb{Z}^{\nu+1},\: 0\leq p\leq \nu$. Then $F^{\beta}(\Psi,a,\Lambda)$ is an irreducible $\tau(B)$-module, which is an evaluation module at a single point. Hence the following result holds as a consequence of S. Eswara Rao and Cuipo Jiang [4].
\begin{theo}
    Let $F^{\beta}(\Psi,a,\Lambda_{0})$ be an irreducible $\tau(B)$-module defined above.
    \begin{itemize}
        \item [(1)] $F^{\beta}(\Psi,a,\Lambda_{0})$ is integrable if and only if $\Lambda_{0}$ is a dominant integral weight of $\mathfrak{\dot g}$.
        \item[(2)] Let $V\in C_{fin}$ be such that $c_{0}=c_{1}=\cdots=c_{\nu}=0$ and $(\mathfrak{\dot g}\otimes A\otimes B)V\neq 0$. then $V\cong F^{\beta}(\Psi,a, \Lambda_{0}),$ for some $(\beta,\Psi,a, \Lambda_{0})$.
        
    \end{itemize}
\end{theo}
\section*{Declarations}
\subsection*{Ethical Approval} Not applicable.
\subsection*{Funding} The authors acknowledge partial funding from Research Apex Project grant to Harish-Chandra Research Institute, Prayagraj.
\subsection*{ Availability of Data and Materials} No data was used for the research described in the article.


\begin{thebibliography}{99}
	\bibitem[1]{}R.V. Moody, S. Eswara Rao, T. Yokonuma, Toroidal Lie algebras and vertex representations, Geom. dedicata
35 (1990) 283–307.
	\bibitem[2]{} Eswara Rao, S and Moody, R.V. Vertex representations for n-toroidal
Lie algebras and a generalization of the Virasoro algebra, Comm. Math.
Phys. 159 (1994), 239-264.
\bibitem[3]{}S. Eswara Rao, Classiﬁcation of irreducible integrable modules for toroidal Lie algebras with ﬁnite
dimensional weight spaces, J. Algebra 277 (1) (2004) 318–348.
\bibitem[4]{} Eswara Rao, S. and Cuipo Jiang, Classifications of irreducible integrable representations for the full toroidal Lie algebra, Journal of Pure
and Applied Algebra, 200(2005), 71-85.
\bibitem[5]{}	Jiang Cuipo and Meng Daoji, Integrable representations for generalized Virasoro-toroidal Lie algebra, Journal of Algebra, 270(2003), 307-334.
\bibitem[6]{}S. Berman, Y. Billig, Irreducible representations for toroidal Lie algebras, Algebra 221 (1999) 188–231.
\bibitem[7]{} Erhard Neher, Alistair Savage, Prasad Senesi, Irreducible ﬁnite-dimensional representations of equiv-
ariant map algebras, Trans. Am. Math. Soc. 364 (5) (2012) 2619–2646.
\bibitem[8]{} S. Eswara Rao, Punita Batra, Classiﬁcation of irreducible integrable highest weight modules for
current Kac-Moody algebras, J. Algebra Appl. 16 (7) (2017) 1750123.
\bibitem[9]{}Sachin S. Sharma, Priyanshu Chakraborty,
Ritesh Kumar Pandey, S. Eswara Rao, Representations of map extended Witt algebras, J. Algebra 639 (2024) 327–353.
\bibitem[10]{}S. Eswara Rao, Partial classification of modules for Lie-algebra of diffeomorphisms of
$d-$dimensional torus.
\bibitem[11]{}P. Chakraborty, S. Eswara Rao, Partial classiﬁcation of irreducible modules for loop-Witt algebras,
J. Lie Theory 32 (2022) 267–279.
\bibitem[12]{}Yuly Billig, Vyacheslav Futorny, Classiﬁcation of irreducible representations of Lie algebra of vector
ﬁelds on a torus, J. Reine Angew. Math. 720 (2016) 199–216.
\bibitem[13]{} S. Eswara Rao, Modules for Loop Affine-Virasoro algebras, Journal of Algebra and Its Applications, Vol.20, No.4,(2021),2150055(13 pages).
\bibitem[14]{}P Chakraborty, P Batra, Classification of irreducible integrable representations for loop toroidal
Lie algebras, arXiv:2007.06415v1.
\bibitem[15]{}Frenkel, I. and Kac, V.G., Basic Representations of affine Lie algebras
and dual resonance models, Invent. Math. 62 (1980), 23-66.
\bibitem[16]{}Alistair Savage, Classiﬁcation of irreducible quasiﬁnite modules over map Virasoro algebras, Trans-
form. Groups 17 (2) (2012) 547–570.
\bibitem[17]{}S. Eswara Rao, Modules for loop aﬃne-Virasoro algebras, J. Algebra Appl. 20 (4) (2021) 2150055.
	\end{thebibliography}
\end{document}